\documentclass[11pt,reqno]{amsart}

\usepackage{amssymb, amsmath, amsthm}
\usepackage{hyperref}
\usepackage[alphabetic,lite]{amsrefs}
\usepackage{verbatim}
\usepackage{amscd}   
\usepackage[all]{xy} 
\usepackage{youngtab} 
\usepackage{young} 
\usepackage{tikz}
\usepackage{ mathrsfs }
\usepackage{cases}
\usepackage{array}

\usepackage[linesnumbered,ruled,vlined]{algorithm2e}
\DontPrintSemicolon

\newcommand{\defi}[1]{{\upshape\sffamily #1}}

\renewcommand{\ll}{\lambda}

\newcommand{\mc}[1]{\mathcal{#1}}
\newcommand{\mf}[1]{\mathfrak{#1}}

\newcommand{\ul}[1]{\underline{#1}}

\newtheorem{theorem}{Theorem}[section]
\newtheorem*{theorem*}{Theorem}
\newtheorem*{problem*}{Problem}
\newtheorem{lemma}[theorem]{Lemma}

\newtheorem{corollary}[theorem]{Corollary}
\newtheorem*{corollary*}{Corollary}

\theoremstyle{definition}
\newtheorem{definition}[theorem]{Definition}
\newtheorem*{definition*}{Definition}
\newtheorem{example}[theorem]{Example}
\newtheorem{problem}[theorem]{Problem}
\newtheorem{question}[theorem]{Question}

\theoremstyle{remark}
\newtheorem{remark}[theorem]{Remark}
\newtheorem*{remark*}{Remark}

\numberwithin{equation}{section}

\begin{document}

\title{Feasibility criteria for high-multiplicity partitioning problems}

\author{Claudiu Raicu}
\address{Department of Mathematics, University of Notre Dame, 255 Hurley, Notre Dame, IN 46556\newline
\indent Institute of Mathematics ``Simion Stoilow'' of the Romanian Academy}
\email{craicu@nd.edu}

\subjclass[2010]{Primary 90C27, 05D99}

\date{\today}

\keywords{Partitioning problem, bin packing, feasibility criteria}

\begin{abstract} For fixed weights $w_1,\cdots,w_n$, and for $d>0$, we let $\mc{B}$ denote a collection of $d\cdot n$ balls, with $d$ balls of weight $w_i$ for each $i=1,\cdots,n$. We consider the problem of assigning the balls to $n$ bins with capacities $C_1,\cdots,C_n$, in such a way that each bin is assigned $d$ balls, without exceeding its capacity. When $d\gg 0$, we give sufficient criteria for the feasibility of this problem, which coincide up to explicit constants with the natural set of necessary conditions. Furthermore, we show that our constants are optimal when the weights~$w_i$ are distinct. The feasibility criteria that we present here are used elsewhere (in commutative algebra) to study the asymptotic behavior of the Castelnuovo--Mumford regularity of symmetric monomial ideals.
\end{abstract}

\maketitle

\section{Introduction}

Partitioning problems, sometimes referred to as the simplest NP-hard problems \cite{hayes}, are fundamental questions in combinatorial optimization with applications to a wide range of resource-allocation problems such as multi-processor scheduling, equitable distribution of assets, voting theory etc. They are concerned with the problem of partitioning a multi-set $\mc{B}$ of resources into parts in order to optimize a certain objective function. The general philosophy (made precise in \cite{mertens}) is that the difficulty of a partitioning problem is controlled by the ratio $M/N$, where $M$ is the maximal size of a resource in $\mc{B}$, and $N$ is the size of $\mc{B}$, with smaller ratios corresponding to simpler problems. We illustrate this philosophy here by studying feasibility conditions for a certain partitioning problem, where $M$ is fixed and $N\to\infty$. The specific problem that we consider is, perhaps surprisingly, motivated by a question in commutative algebra, namely that of describing the asymptotic behavior of the Castelnuovo--Mumford regularity of powers of monomial ideals that are invariant under the symmetric group of permutations of the coordinates (see \cite[Section~5]{raicu-sn-invariant} for an explanation of this connection). As we will see, the problem is quite natural from the point of view of combinatorial optimization. It fits in the realm of high-multiplicity optimization problems, that have been studied extensively in operations research and computer science \cites{fernandez-lueker,hoc-sha,mc-sma-spi,fil-agn,goe-rot}.

We fix a positive integer $n$, and a tuple of non-negative integers $\ul{w} = (w_1,\cdots,w_n)$, which we call \defi{weights}. For a positive integer $d$ and a tuple of integer \defi{capacities} $\ul{C} = (C_1,\cdots,C_n)$, we consider the \defi{partitioning problem $\mf{BP}(d,\ul{C};\ul{w})$} defined as follows: given a collection of $d\cdot n$ balls, with $d$ of weight $w_i$ for each $i=1,\cdots,n$, and given bins $\mc{B}_1,\cdots,\mc{B}_n$ with capacities $C_1,\cdots,C_n$, determine if there exists an assignment of the balls to bins such that for each $i$ we have that $\mc{B}_i$ contains exactly $d$ balls, and the total weight of the balls in $\mc{B}_i$ does not exceed $C_i$. We say that $\mf{BP}(d,\ul{C};\ul{w})$ is \defi{feasible} if such an assignment exists. An equivalent formulation of the problem is obtained by considering the multiset of ball-weights
\[\mc{B} = \{w_1,\cdots,w_1,\cdots,w_i,\cdots,w_i,\cdots,w_n,\cdots,w_n\},\]
where each $w_i$ is repeated $d$ times, and asking whether there exists a multi-set partition
\begin{equation}\label{eq:partition-B}
\mc{B} = \mc{B}_1 \sqcup \cdots \sqcup \mc{B}_i \sqcup \cdots \sqcup \mc{B}_n,
\end{equation}
such that each $\mc{B}_i$ has exactly $d$ elements, and
\begin{equation}\label{eq:constraints-partitioning}
w(\mc{B}_i):=\sum_{w\in\mc{B}_i} w \leq C_i,\mbox{ for }i=1,\cdots,n.
\end{equation}
Any partition $\mc{B}_{\bullet}$ satisfying the constraints (\ref{eq:constraints-partitioning}) is said to be \defi{feasible}. Throughout this article we will use interchangeably the bin and multi-set partitioning terminology. Our goal is to find sufficient conditions for the feasibility of $\mf{BP}(d,\ul{C};\ul{w})$ when $d\gg 0$, and to explain the extent to which these conditions are optimal. 

We begin with an example of a feasible partitioning problem that we will return to throughout this article:

\begin{example}\label{ex:feasible-BP}
  Let $n=6$, $\ul{w}=(5,5,3,1,1,0)$, $d=6$, and $\ul{C}=(17,17,17,17,17,8)=(17^5,8)$. The following partition $\mc{B}_{\bullet}$ gives a solution to $\mf{BP}(d,\ul{C};\ul{w})$: the notation $w_1^{a_1}w_2^{a_2}\cdots$ means that we use $a_1$ balls of weight $w_1$, $a_2$ balls of weight $w_2$, etc. Note that since $w_1=w_2=5$, the total number of balls of weight $5$ is $2\cdot d=12=1+2+3+3+3+0$; a similar comment applies to $w_4=w_5=1$.
\[
\setlength{\extrarowheight}{2pt}
\begin{array}{c|c|c|c|c|c|c}
a & 1 & 2  & 3 & 4 & 5 & 6 \\
 \hline 
\mc{B}_a & 5^1 3^3 1^2 & 5^2 3^2 1^1 0^1 & 5^3 1^2 0^1 & 5^3 1^1 0^2 & 5^3 1^2 0^1 & 3^1 1^4 0^1\\ 
 \hline 
w(\mc{B}_a) & 16 & 17 & 17  & 16 & 17 & 7 \\ 
\end{array}
\]
\end{example}

\noindent The simplest case when the feasibility of $\mf{BP}(d,\ul{C};\ul{w})$ can be characterized is when all the weights are equal: 

\begin{example}\label{ex:n=1}
 When $n=1$, we have $\ul{w}=(w)$ and $\ul{C}=(C)$ are singletons, and a necessary and sufficient condition for the feasibility of $\mf{BP}(d,\ul{C};\ul{w})$ is given by the inequality $C\geq d\cdot w$. Suppose more generally that $n\geq 1$ and $w_1=\cdots=w_n=w$. Any solution to $\mf{BP}(d,\ul{C};\ul{w})$ will assign to each $\mc{B}_i$ exactly $d$ balls of weight $w$. This is feasible if and only if $C_i\geq d\cdot w$ for all $i=1,\cdots,n$.
\end{example}

It will be useful from now on to order the weights and capacities, so we will assume that
\[ w_1\geq\cdots\geq w_n\quad\mbox{ and }\quad C_1\geq\cdots\geq C_n.\]
One has that if $\mf{BP}(d,\ul{C};\ul{w})$ is feasible then 
\begin{equation}\label{eq:domin-C-w}
 C_i + C_{i+1} + \cdots + C_n \geq d\cdot(w_i + \cdots + w_n)\mbox{ for each }i=1,\cdots,n.
\end{equation}
This follows since for a feasible partition $\mc{B}_{\bullet}$, the bins $\mc{B}_i,\cdots,\mc{B}_n$ must contain collectively a total of $d\cdot (n-i+1)$ balls, whose total weight can be no smaller than the sum of the smallest $d\cdot(n-i+1)$ elements of the multi-set $\mc{B}$, namely $d\cdot(w_i + \cdots + w_n)$. It is an interesting consequence of our main result below (Theorem~\ref{thm:sufficient-condition}) that the conditions (\ref{eq:domin-C-w}) are also sufficient to guarantee the feasibility of $\mf{BP}(d,\ul{C};\ul{w})$ when $d\gg 0$ and the weights $\ul{w}$ are \defi{balanced}, that is, when $0\leq w_i-w_{i+1}\leq 1$ for all $i=1,\cdots,n-1$. In commutative algebra, balanced weights give rise to ideals that have the remarkable property that their powers have eventually a linear minimal free resolution \cite[Section~5.3]{raicu-sn-invariant}. 

In general, Theorem~\ref{thm:sufficient-condition} provides constant correction factors that transform (\ref{eq:domin-C-w}) into sufficient conditions for feasibility. To see that conditions (\ref{eq:domin-C-w}) cannot be sufficient in general, consider the following.

\begin{example}\label{ex:not-suff-n=2}
 Let $n=2$, $\ul{w}=(3,1)$, and for $d=2$ consider $\ul{C} = (5,3)$, so that (\ref{eq:domin-C-w}) holds. However, since $w_1+w_2>C_2$, it follows that any solution of $\mf{BP}(d,\ul{C};\ul{w})$ can only place balls of weight $w_2=1$ into $\mc{B}_2$, so that $\mc{B}_2=\{1,1\}$, which then forces $\mc{B}_1=\{3,3\}$, exceeding the capacity $C_1=5$. We conclude that $\mf{BP}(d,\ul{C};\ul{w})$ is not feasible in this case. In fact, for any positive integer $d$ we can let $\ul{C} = (3d-1,d+1)$ to obtain an infeasible problem $\mf{BP}(d,\ul{C};\ul{w})$. One can also check that for $\ul{C}=(2d,2d)$, the problem $\mf{BP}(d,\ul{C};\ul{w})$ is feasible when $d$ is even, and infeasible when $d$ is odd (see also Example~\ref{ex:suff-n=2}).
\end{example}

One can check that replacing $w_i$ by $w_i-w_n$ and $C_i$ by $C_i-d\cdot w_n$ leads to an equivalent partitioning problem. Since it doesn't play a major role in our arguments we won't make this reduction here, but we note that whenever we bound $d$ in terms of the highest weight $w_1$, one can in fact improve the bound by considering instead the difference $w_1-w_n$ between the highest and lowest weight.

To state our results we need to introduce some notation. A tuple $\ll=(\ll_1,\cdots,\ll_r)$ of non-negative integers with non-increasing entries $\ll_1\geq\ll_2\geq\cdots\geq\ll_r$ is called a \defi{partition} and it is pictured in the form of a \defi{Young diagram} of left-justified boxes, with $\ll_i$ boxes in row $i$. For instance, $\ll=(4,2,1)$ will be pictured as
\begin{equation}\label{eq:yng421}
\Yvcentermath1\yng(4,2,1)
\end{equation}
The only ambiguity in recovering $\ll$ from its diagram is the number of trailing zeros, as (\ref{eq:yng421}) is for instance also the Young diagram of $\mu=(4,2,1,0,0)$. The \defi{conjugate partition} of $\ll$ is denoted $\ll'$, and is obtained by transposing the corresponding Young diagram. For instance, if $\ll=(4,2,1)$ then $\ll'=(3,2,1,1)$. For partitions with repeating parts, we use the abbreviation $(b^a)$ to denote the sequence $(b,b,\cdots,b)$ of length $a$; for instance $(3,3,3,3,1,1)$ will be abbreviated as $(3^4,1^2)$.

\begin{definition}\label{def:b-lam}
 Consider a partition $\ll=(\ll_1,\cdots,\ll_r)$ and write $\ll'=(r^{a_0},h_1^{a_1},h_2^{a_2},\cdots,h_k^{a_k})$, with $a_0\geq 0$ and $a_1,\cdots,a_k>0$, $r>h_1>\cdots>h_k>0$. Note that $a_0=0$ if and only if $\ll_r=0$, and that $k=0$ if and only if $\ll_1=\cdots=\ll_r$. We define $b(\ll)=0$ if $k=0$, and otherwise let
 \[
 \begin{aligned}
 b(\ll) &= (r-h_1)\cdot(a_1-1) + (h_1-h_2)\cdot(a_2-1) + \cdots + (h_{k-1}-h_k)\cdot(a_k-1) + (h_k-1)\cdot(a_k-1) \\
 &= \left(\sum_{t=1}^k (h_{t-1}-h_t)\cdot(a_t-1)\right) + (h_k-1)\cdot (a_k-1).
 \end{aligned}
 \]
 where in the second equality we set $h_0=r$.
\end{definition}

\begin{remark}\label{rem:b-lam}
 Note that $b(\ll)$ will typically change when we add trailing zeroes to $\ll$. Note also that when $\ll_1,\cdots,\ll_r$ are distinct, we have $k=r-1$, $h_i=r-i$ for $i=1,\cdots,r-1$, and $b(\ll)=\ll_1-\ll_r-r+1$. For an alternative calculation of $b(\ll)$, see Lemma~\ref{lem:bw=sum-gaps}.
\end{remark}

\begin{example}\label{ex:b-lam}
 (a) Suppose that $\ll=(3,3,0)$, so that $r=3$. We have that $\ll'=(2^3)$ and $b(\ll)=4$. If we consider instead $\mu=(3,3)$ then $r=2$, $\mu'=(2^3)=(r^3)$ and $b(\mu)=0$.
 
 \noindent (b) Suppose that $\ll=(5,5,3,1,1,0)$, so that $r=6$. We have that $\ll'=(5^1,3^2,2^2)$ and $b(\ll) = 4$. The formula for $b(\ll)$ is not affected in this case by trailing zeroes, and in particular $b(\mu)=4$ when $\mu=(5,5,3,1,1)$.
\end{example}

We will apply Definition~\ref{def:b-lam} to $\ll=\ul{w}^{\geq i}$ a truncation of the vector of weights $\ul{w}=(w_1,\cdots,w_n)$, where
\[ \ul{w}^{\geq i} = (w_i,w_{i+1},\cdots,w_n) \mbox{ for }i=1,\cdots,n.\]
We are now in the position to state the main result.

\begin{theorem}\label{thm:sufficient-condition}
 Fix $n$ and a tuple $\ul{w}=(w_1\geq\cdots\geq w_n)$ of non-negative weights. There exists a positive integer $d_{\ul{w}}$ such that for every $d\geq d_{\ul{w}}$ and every tuple of capacities $\ul{C}=(C_1\geq\cdots\geq C_n)$ satisfying
 \begin{equation}\label{eq:sufficient-cond-C-w}
 C_i + C_{i+1} + \cdots + C_n \geq d\cdot(w_i + \cdots + w_n) + b(\ul{w}^{\geq i})\mbox{ for each }i=1,\cdots,n,
\end{equation}
the partitioning problem $\mf{BP}(d,\ul{C};\ul{w})$ is feasible. Moreover, one can take $d_{\ul{w}} = n^3\cdot w_1\cdot(2\cdot n+w_1)$. 

If the weights $w_i$ are distinct and if we relax any of the inequalities (\ref{eq:sufficient-cond-C-w}) then there exists $\ul{C}$ satisfying the relaxed conditions for which $\mf{BP}(d,\ul{C};\ul{w})$ is infeasible.
\end{theorem}

Notice that the integers $b(\ul{w}^{\geq i})$ in (\ref{eq:sufficient-cond-C-w}) depend only on $\ul{w}$ and not on $d$. They provide appropriate correction factors to the necessary (but non-sufficient) conditions (\ref{eq:domin-C-w}) to guarantee feasability. We were not able to verify that the integers $b(\ul{w}^{\geq i})$ are optimal for arbitrary weights $w_i$, but we do not know any example when they are not (see also Theorem~\ref{thm:optimal-bw}, Remark~\ref{rem:when-C-is-C0} and Section~\ref{sec:open}). When $\ul{w}$ is balanced, the truncations are also balanced, and one can check that as a consequence $b(\ul{w}^{\geq i})=0$ for all $i$, so the conditions (\ref{eq:domin-C-w}) and (\ref{eq:sufficient-cond-C-w}) become identical. It follows that the conditions (\ref{eq:sufficient-cond-C-w}) are both necessary and sufficient, and in particular the constants $b(\ul{w}^{\geq i})$ are optimal in this case.

\begin{example}\label{ex:suff-n=2}
 As in Example~\ref{ex:not-suff-n=2}, let $n=2$ and $\ul{w}=(3,1)$. We have $\ul{w}^{\geq 1}=\ul{w}$, $\ul{w}^{\geq 2}=(1)$, and $b(\ul{w}^{\geq 1})=1$, $b(\ul{w}^{\geq 2})=0$, so the conditions in (\ref{eq:sufficient-cond-C-w}) become
 \[ C_1 + C_2 \geq 4d+1,\quad C_2\geq d,\]
 which are satisfied for instance by $\ul{C}=(2d+1,2d)$. When $d$ is even, we can then solve $\mf{BP}(d,\ul{C};\ul{w})$ by distributing the balls evenly to $\mc{B}_1$ and $\mc{B}_2$. When $d$ is odd, we place $(d+1)/2$ balls of weight $w_1$, and $(d-1)/2$ balls of weight $w_2$ in $\mc{B}_1$, and place the remaining $d$ balls in $\mc{B}_2$.
\end{example}

\begin{example}\label{ex:optimality-main-thm}
 Note that the conditions (\ref{eq:sufficient-cond-C-w}) are not necessary for $\mf{BP}(d,\ul{C};\ul{w})$ to be feasible. For instance, if $d=6$, $\ul{w}=(5,5,3,1,1,0)$, and $\ul{C}=(17^5,8)$ as in Example~\ref{ex:feasible-BP}, the resulting problem is feasible, despite the fact that (\ref{eq:sufficient-cond-C-w}) is violated for $i=1$: we have $C_1+\cdots+C_n=93$, $d\cdot(w_1+\cdots+w_n)=90$, and $b(\ul{w}) = 4$. See Example~\ref{ex:optimal-1-feasible} for a closely related, but infeasible problem.
\end{example}

To prove Theorem~\ref{thm:sufficient-condition} we set up an inductive procedure, based on the number of bins. We consider partial relaxations of the feasibility condition: we say that the problem $\mf{BP}(d,\ul{C};\ul{w})$ is \defi{$k$-feasible} if there exists an assignment of balls to bins (with $d$ balls in each bin) such that
\[ w(\mc{B}_i) \leq C_i,\mbox{ for }i=k+1,\cdots,n,\]
in which case $\mc{B}_{\bullet}$ is called \defi{$k$-feasible}. Equivalently, when looking for a solution for $\mf{BP}(d,\ul{C};\ul{w})$ we allow the capacities to be exceeded in bins $\mc{B}_1,\cdots,\mc{B}_k$. Furthermore, if we let $\ul{D} = (\infty,\cdots,\infty,C_{k+1},\cdots,C_n)$, then $\mf{BP}(d,\ul{C};\ul{w})$ is $k$-feasible if and only if $\mf{BP}(d,\ul{D};\ul{w})$ is feasible. It is clear that $(k-1)$-feasible problems are also $k$-feasible. The interesting direction is then to understand the additional conditions which imply that a $k$-feasible problem is also $(k-1)$-feasible. To that end, we prove the following (note that $0$-feasible is the same as feasible).

\begin{theorem}\label{thm:inductive-feasibility}
 Fix $n$ and a tuple $\ul{w}=(w_1\geq\cdots\geq w_n)$ of non-negative weights. There exists a positive integer $d^1_{\ul{w}}$ such that for every $d\geq d^1_{\ul{w}}$ and every tuple of capacities $\ul{C}=(C_1\geq\cdots\geq C_n)$ for which $\mf{BP}(d,\ul{C};\ul{w})$ is $1$-feasible, we have that if in addition
 \begin{equation}\label{eq:size-C-feasibility}
 C_1 + \cdots + C_n \geq d\cdot(w_1 + \cdots + w_n) + b(\ul{w})
\end{equation}
then $\mf{BP}(d,\ul{C};\ul{w})$ is feasible. Moreover, one can take $d^1_{\ul{w}} = n^3\cdot w_1\cdot(2\cdot n+w_1)$.
\end{theorem}

The proof strategy behind Theorem~\ref{thm:inductive-feasibility}, outlined in Section~\ref{sec:strategy-inductive-feasibility}, is to start with a $1$-feasible partition $\mc{B}_{\bullet}$ and then perform a suitable sequence of ball exchanges leading to a feasible partition. This idea is not new, as exchange algorithms are known to give useful heuristics for solving partitioning problems (see for instance \cite[Section~3.3]{babel}), but the details in our specific case are somewhat involved. It is perhaps refreshing to know that Hall's Marriage Theorem provides one of the key steps in our argument (Section~\ref{sec:Hall-Marriage}).

When $w_1=\cdots=w_n=w$ and $n\geq 2$, knowing that $\mf{BP}(d,\ul{C};\ul{w})$ is $1$-feasible forces $C_n\geq d\cdot w$, so $C_i\geq d\cdot w$ for all $i$, which means that $\mf{BP}(d,\ul{C};\ul{w})$ is in fact feasible. Theorem~\ref{thm:inductive-feasibility} is therefore not interesting unless $\ul{w}$ has at least two distinct parts. The next theorem shows that (\ref{eq:size-C-feasibility}) is not superfluous in this general case, and more importantly, it shows that the constant $b(\ul{w})$ is optimal!

\begin{theorem}\label{thm:optimal-bw}
 Suppose that $\ul{w}=(w_1,\cdots,w_n)$ has at least two distinct parts and that $d\geq n$. There exists a sequence $\ul{C}$ with
 \begin{equation}\label{eq:opt-1}
  C_1+\cdots+C_n = d\cdot(w_1+\cdots+w_n) + b(\ul{w}) - 1
 \end{equation}
 such that $\mf{BP}(d,\ul{C};\ul{w})$ is $1$-feasible but not feasible.
\end{theorem}

Notice the discrepancy between the lower bound $d\geq n$ in Theorem~\ref{thm:optimal-bw}, and the formula that we give for~$d^1_{\ul{w}}$ in Theorem~\ref{thm:inductive-feasibility}. It would be interesting to understand what the optimal value of $d^1_{\ul{w}}$ is that makes Theorem~\ref{thm:inductive-feasibility} true, and in particular whether it can be taken to only depend on $n$ and not on $\ul{w}$.

\bigskip

\noindent{\bf Organization.} The proof of Theorem~\ref{thm:inductive-feasibility} will occupy most of this article. We outline the general strategy in Section~\ref{sec:strategy-inductive-feasibility}, and verify some of the simple steps, while in Sections~\ref{sec:many-w1},~\ref{sec:Hall-Marriage},~\ref{sec:shrink-gaps} we explain the more substantial steps of our argument. In Section~\ref{sec:optimality} we prove Theorem~\ref{thm:optimal-bw} explaining the optimality of the constant $b(\ul{w})$. We prove Theorem~\ref{thm:sufficient-condition} in Section~\ref{sec:feasibility}, and conclude with some open questions in Section~\ref{sec:open}. Throughout, we illustrate our results with concrete examples in order to make the arguments more transparent.


\section{Proof strategy for Theorem~\ref{thm:inductive-feasibility}}\label{sec:strategy-inductive-feasibility}

In this section we fix some terminology to be used throughout the article, we outline the proof of Theorem~\ref{thm:inductive-feasibility}, and we verify the easier steps in the argument. Given a partition $\mc{B}_{\bullet}$ as in (\ref{eq:partition-B}), we consider for $1\leq i,j\leq n$ the number $n_i(\mc{B}_j)$ of balls of weight $w_i$ in~$\mc{B}_j$.

\begin{example}\label{ex:ni-Bj}
 For the partition $\mc{B}_{\bullet}$ from Example~\ref{ex:feasible-BP}, we have the following table recording in row $i$ and column $j$ the number $n_i(\mc{B}_j)$:
\[
\setlength{\extrarowheight}{2pt}
\begin{array}{c|cccccc}
 & 1 & 2  & 3 & 4 & 5 & 6 \\
 \hline 
1 & 1 & 2 & 3  & 3 & 3 & 0 \\ 
2 & 1 & 2 & 3  & 3 & 3 & 0 \\ 
3 & 3 & 2 & 0  & 0 & 0 & 1 \\ 
4 & 2 & 1 & 2  & 1 & 2 & 4 \\ 
5 & 2 & 1 & 2  & 1 & 2 & 4 \\ 
6 & 0 & 1 & 1  & 2 & 1 & 1 \\ 
\end{array}
\]
\end{example}

We define the \defi{gap sequence} $g_{\bullet}(\mc{B}_{\bullet})$ of the partition $\mc{B}_{\bullet}$ via
\begin{equation}\label{eq:def-gi-B}
 g_i(\mc{B}_{\bullet}) = C_i - w(\mc{B}_i)\mbox{ for }i=1,\cdots,n.
\end{equation}
One can then reinterpret the condition that $\mc{B}_{\bullet}$ is $k$-feasible by the inequalities $g_i(\mc{B}_{\bullet}) \geq 0$ for $i>k$, and in particular a feasible $\mc{B}_{\bullet}$ is one for which all gaps are non-negative.

\begin{example}
 If we take $\ul{C} = (17^5,8)$ and $\mc{B}_{\bullet}$ as in Example~\ref{ex:feasible-BP} then we get the gap sequence
 \[
\setlength{\extrarowheight}{2pt}
\begin{array}{c|cccccc}
a & 1 & 2  & 3 & 4 & 5 & 6 \\
 \hline 
g_a(\mc{B}_{\bullet}) & 1 & 0 & 0  & 1 & 0 & 1 \\ 
\end{array}
\]
\end{example}

Notice that condition (\ref{eq:size-C-feasibility}) gives a lower bound for the sum of the gaps of a partition $\mc{B}_{\bullet}$, which leads to the following quick feasibility criterion.

\begin{lemma}\label{lem:small-sum-gi}
 Suppose that (\ref{eq:size-C-feasibility}) holds and that $\mc{B}_{\bullet}$ is $1$-feasible. If
 \begin{equation}\label{eq:small-sum-gi}
 \sum_{i=2}^n g_i(\mc{B}_{\bullet}) \leq b(\ul{w})
 \end{equation}
 then $\mc{B}_{\bullet}$ is feasible.
\end{lemma}

\begin{proof}
 We have that
 \[ g_1(\mc{B}_{\bullet}) = \sum_{i=1}^n \bigl(C_i - w(\mc{B}_i)\bigr) - \sum_{i=2}^n g_i(\mc{B}_{\bullet}) \geq b(\ul{w}) - \sum_{i=2}^n g_i(\mc{B}_{\bullet}) \geq 0,\]
where the first inequality follows from (\ref{eq:size-C-feasibility}) and the fact that $\sum_{i=1}^n w(\mc{B}_i) = d\cdot(w_1+\cdots+w_n)$, while the second inequality is by hypothesis. Since $\mc{B}_{\bullet}$ is $1$-feasible and $g_1(\mc{B}_{\bullet})\geq 0$, we get that $\mc{B}_{\bullet}$ is feasible.
\end{proof}

Based on Lemma~\ref{lem:small-sum-gi}, the key idea behind the proof of Theorem~\ref{thm:inductive-feasibility} is to look for $1$-feasible partitions with small gaps $g_i(\mc{B}_{\bullet})$ for $i\geq 2$. The precise proof strategy is based on the following outline, to be detailed in the subsequent sections.

\begin{itemize}
 \item[{\bf Step 1.}] We consider all the $1$-feasible partitions $\mc{B}_{\bullet}$ for which the weight $w(\mc{B}_1)$ takes the minimal value, and denote this value by $W^{min}$. If $W^{min}\leq C_1$ then any such partition is in fact feasible, so $\mf{BP}(d,\ul{C};\ul{w})$ is feasible. We suppose that $W^{min}>C_1$ and seek a contradiction in the following steps.
 \item[{\bf Step 2.}] If a partition $\mc{B}_{\bullet}$ has $n_1(\mc{B}_1)\neq 0$ and $g_i(\mc{B}_{\bullet})\geq w_1$ for some $i\geq 2$, we show that a single ball swap creates a $1$-feasible partition $\mc{B}'_{\bullet}$ with $w(\mc{B}'_1)<w(\mc{B}_1)$, contradicting the minimality of $w(\mc{B}_1)$.
 \item[{\bf Step 3.}] Among all the partitions $\mc{B}_{\bullet}$ considered in {\bf Step 1}, we consider one for which $n_1(\mc{B}_1)$ is maximal. We show that if $d$ is large, then $n_1(\mc{B}_1)$ is large as well. In particular $n_1(\mc{B}_1)$ is non-zero, so by {\bf Step~2} we may assume that all gaps satisfy $g_i(\mc{B}_{\bullet})<w_1$.
 \item[{\bf Step 4.}] With $\mc{B}_{\bullet}$ as in {\bf Step 3}, we show that if $d$ is large then we can find a permutation $R_{\bullet}$ of $\{1,2,\cdots,n\}$ with $R_1=1$, and with the property that $n_i(\mc{B}_{R_i})$ is large for each $i=1,2,\cdots,n$.
 \item[{\bf Step 5.}] Using Algorithm~\ref{alg:gaps} (the ``Shrinking gaps algorithm"), we show that through a series of ball swaps we can reach a $1$-feasible partition $\mc{B}'_{\bullet}$ whose gaps $g_i(\mc{B}'_{\bullet})$ are small for $i\geq 2$. Based on Lemma~\ref{lem:small-sum-gi} we deduce that $\mc{B}'_{\bullet}$ is feasible, and therefore $w(\mc{B}_1')\leq C_1<W^{min}=w(\mc{B}_1)$, a contradiction.
\end{itemize}

The set of $1$-feasible partitions is non-empty by the hypothesis of Theorem~\ref{thm:inductive-feasibility}, so {\bf Step 1} of the outline requires no further explanations. We record an important consequence of the inequality $W^{min}>C_1$.

\begin{lemma}\label{lem:Bi-not-only-w1}
 Suppose that $\mc{B}_{\bullet}$ is $1$-feasible and $w(\mc{B}_1)>C_1$ (which is automatic if we assume $W^{min}>C_1$). For each $i\geq 2$ we have that $\mc{B}_i$ contains some ball of weight different from $w_1$.
\end{lemma}

\begin{proof}
 If for some $i\geq 2$ we have that $\mc{B}_i$ consists of $d$ balls of weight $w_1$ then
 \[ w(\mc{B}_1) > C_1 \geq C_i \geq w(\mc{B}_i) = d\cdot w_1,\]
 Since $\mc{B}_1$ contains $d$ balls, each of weight $\leq w_1$, we have $w(\mc{B}_1)\leq d\cdot w_1$, contradicting the inequality above.
\end{proof}

We next explain {\bf Step 2} of the outline, which is a direct consequence of the following.

\begin{lemma}\label{lem:swap-B1-w1}
 Suppose that $W^{min}>C_1$, $\mc{B}_{\bullet}$ is $1$-feasible, $n_1(\mc{B}_1)\neq 0$ and $g_i(\mc{B}_{\bullet})\geq w_1$ for some $i\geq 2$. There exists a $1$-feasible partition $\mc{B}'_{\bullet}$ with $w(\mc{B}'_1)<w(\mc{B}_1)$.
\end{lemma}

\begin{proof}
 By Lemma~\ref{lem:Bi-not-only-w1}, there exists a ball of weight $w_j\neq w_1$ in $\mc{B}_i$, so $w_j<w_1$. Consider the partition $\mc{B}'_{\bullet}$ obtained from $\mc{B}_{\bullet}$ by swapping a ball of weight $w_1$ in $\mc{B}_1$ (which exists since $n_1(\mc{B}_1)\neq 0$) with a ball of weight $w_j$ from $\mc{B}_i$. Note that $g_a(\mc{B}'_{\bullet})=g_a(\mc{B}_{\bullet})\geq 0$ for all $a\neq 1,i$. Moreover, we have
 \[g_i(\mc{B}'_{\bullet}) = g_i(\mc{B}_{\bullet}) - w_1 + w_j \geq w_j \geq 0,\]
 so $\mc{B}'_{\bullet}$ is $1$-feasible. Finally,
 \[ w(\mc{B}'_1) = w(\mc{B}_1) - w_1 + w_j < w(\mc{B}_1),\]
 as desired.
\end{proof}

The rest of the argument requires more work: we check {\bf Step 3} in Section~\ref{sec:many-w1}, we use Hall's Marriage Theorem to deduce {\bf Step 4} in Section~\ref{sec:Hall-Marriage}, and explain the ``Shrinking gaps algorithm" giving {\bf Step 5} in Section~\ref{sec:shrink-gaps}. As explained in the Introduction, Theorem~\ref{thm:inductive-feasibility} is easy when $w_1=\cdots=w_n$, so we will assume when needed (in Section~\ref{sec:shrink-gaps}) that $\ul{w}$ has at least two distinct parts (in particular $n\geq 2$). We caution the reader that we will be quite relaxed with our estimates below, for the sake of clarity and at the cost of finding a (potentially) far from optimal constant $d^1_{\ul{w}}$.

\section{Many balls of weight $w_1$ in $\mc{B}_1$}\label{sec:many-w1}

The goal of this section is to prove the following result making effective the notion of ``large" in {\bf Step~3} of our outline from Section~\ref{sec:strategy-inductive-feasibility}. The proof is based on an exchange procedure that involves several balls, and is illustrated in Example~\ref{ex:many-w1-in-B1} at the end of this section. We write $|\ul{w}|=w_1+\cdots+w_n$.

\begin{theorem}\label{thm:many-w1}
 Suppose that $\mf{BP}(d,\ul{C};\ul{w})$ is $1$-feasible and $W^{min}>C_1$, and fix a positive integer $N>0$. If
\begin{equation}\label{eq:bound-d-manyw1}
 d \geq n\cdot N + n\cdot(n-1)\cdot w_1 \cdot |\ul{w}|
\end{equation}
then there exists a $1$-feasible partition $\mc{B}_{\bullet}$ with $w(\mc{B}_1)=W^{min}$ and $n_1(\mc{B}_1)\geq N$.
\end{theorem}

\begin{proof}
Among all the $1$-feasible partitions $\mc{B}_{\bullet}$ with $w(\mc{B}_1)=W^{min}$, we choose one for which $n_1(\mc{B}_1)$ is maximal. If $n_1(\mc{B}_1)\geq N$ then we are done. Suppose by contradiction that $n_1(\mc{B}_1)<N$. Note that
\[\sum_{j=1}^n n_1(\mc{B}_j) \geq d,\]
since there are (at least) $d$ balls of weight $w_1$. Combined with $n_1(\mc{B}_1)<N$, this implies 
\begin{equation}\label{eq:bd-n1-Br}
n_1(\mc{B}_r) > \frac{d-N}{n-1}\mbox{ for some }r\geq 2.
\end{equation}
We reach a contradiction with the maximality of $n_1(\mc{B}_1)$ in three steps, as follows.

\noindent{\bf Step 3.1.} We claim that there exists $k>1$ with $w_k<w_1$ and $n_k(\mc{B}_1)\geq w_1$. If this wasn't the case, then each weight $w_j\neq w_1$ would appear at most $(n-1)$ times in $\mc{B}_1$, and since $n_1(\mc{B}_1)<N$, we would get
\[ d = |\mc{B}_1| < N + (n-1) w_1,\]
contradicting (\ref{eq:bound-d-manyw1}). We may assume that $k$ is minimal, so for $j<k$ either $n_j(\mc{B}_1)\leq w_1-1$, or $w_j=w_1$. Since $w_j\leq w_k$ for $j\geq k$, we obtain
\begin{equation}\label{eq:ub-wB1}
w(\mc{B}_1) \leq N\cdot w_1 + (w_1-1)\cdot(w_2+\cdots+w_{k-1}) + (d-N)\cdot w_k \leq N\cdot (w_1-w_k) + d\cdot w_k + w_1\cdot |\ul{w}|.
\end{equation}

\noindent{\bf Step 3.2.} We next claim that there exists $t>k$ with $w_t<w_k$ and $n_t(\mc{B}_r)\geq w_1$, where $k$ is as in {\bf Step 3.1} and $r$ is as in (\ref{eq:bd-n1-Br}). Suppose by contradiction that this isn't the case, and note that $w_t\geq w_k$ for $t<k$, so
\begin{equation}\label{eq:lb-wBr}
\begin{aligned}
w(\mc{B}_r) &\geq n_1(\mc{B}_r)\cdot w_1 + (d-n_1(\mc{B}_r)-(n-k)\cdot(w_1-1))\cdot w_k + (w_1-1)\cdot(w_{k+1}+\cdots+w_n)  \\
&\geq n_1(\mc{B}_r)\cdot (w_1-w_k) + d\cdot w_k - (n-1)\cdot w_1\cdot |\ul{w}|.\\
\end{aligned}
\end{equation}
Since $w(\mc{B}_1)>C_1\geq C_r\geq w(\mc{B}_r)$, it follows from (\ref{eq:ub-wB1}) and (\ref{eq:lb-wBr}) that 
\[n_1(\mc{B}_r)\cdot (w_1-w_k) + d\cdot w_k - (n-1)\cdot w_1\cdot |\ul{w}| < N\cdot (w_1-w_k) + d\cdot w_k + w_1\cdot |\ul{w}|.\]
Rewriting this inequality, and combining it with (\ref{eq:bd-n1-Br}), it follows that
\[\frac{d-N}{n-1} < n_1(\mc{B}_r) < N + \frac{n\cdot w_1\cdot |\ul{w}|}{w_1-w_k} \leq N + n\cdot w_1\cdot |\ul{w}|,\]
which implies that $d<n\cdot N + n\cdot(n-1)\cdot w_1\cdot |\ul{w}|$, contradicting (\ref{eq:bound-d-manyw1}).

\noindent{\bf Step 3.3.} By the previous steps, we know that $\mc{B}_1$ contains at least $w_1$ balls of weight $w_k<w_1$, that $\mc{B}_r$ contains at least $w_1$ balls of weight $w_t<w_k$, and that $\mc{B}_r$ also contains
\[ n_1(\mc{B}_r) > \frac{d-N}{n-1} \geq N + n\cdot w_1 \cdot |\ul{w}| \geq w_1\]
balls of weight $w_1$. We can then move
\begin{itemize}
 \item $(w_1-w_t)$ balls of weight $w_k$ from $\mc{B}_1$ to $\mc{B}_r$.
 \item $(w_1-w_k)$ balls of weight $w_t$ from $\mc{B}_r$ to $\mc{B}_1$.
 \item $(w_k-w_t)$ balls of weight $w_1$ from $\mc{B}_r$ to $\mc{B}_1$.
\end{itemize}
Since $(w_1-w_t) = (w_1-w_k)+(w_k-w_t)$ and 
\[(w_1-w_t)\cdot w_k = (w_1-w_k)\cdot w_t + (w_k-w_t)\cdot w_1,\]
it follows that  the number of balls in $\mc{B}_1,\mc{B}_r$ is unchanged (namely $d$), and that $w(\mc{B}_1)$ and $w(\mc{B}_r)$ are also unchanged, so the resulting partition is still $1$-feasible with minimal $w(\mc{B}_1)$. However, the number of balls of weight $w_1$ in $\mc{B}_1$ has increased by $w_k-w_t>0$, which contradicts the maximality of $n_1(\mc{B}_1)$.
\end{proof} 
 
\begin{example}\label{ex:many-w1-in-B1}
 The condition (\ref{eq:bound-d-manyw1}) is sufficient (but not necessary) to guarantee the existence of a partition $\mc{B}_{\bullet}$ with many balls of weight $w_1$ in $\mc{B}_1$. Likewise the estimates that we use for the number of balls in various bins can often be improved. In the example below we only illustrate the exchange in {\bf Step 3.3} above (without worrying about our estimates or the bound on $d$ being satisfied). We take $\ul{w}=(5,5,3,1,1,0)$, $d=6$, $\ul{C}=(17^5,8)$, $N=2$, and the partition $\mc{B}_{\bullet}$ given by
\[
\setlength{\extrarowheight}{2pt}
\begin{array}{c|c|c|c|c|c|c}
a & 1 & 2  & 3 & 4 & 5 & 6 \\
 \hline 
\mc{B}_a & 3^6 & 5^3 1^2 0^1 & 5^3 1^2 0^1 & 5^3 1^1 0^2 & 5^3 1^1 0^2 & 1^6 \\ 
\end{array}
\]
We take $k=3$, and note that $\mc{B}_1$ contains six balls of weight $w_k=3$. We take $r=2$ and note that $\mc{B}_r$ contains three balls of weight $w_1=5$. We take $t=4$ and note that $\mc{B}_r$ contains two balls of weight $w_t=1$. If we move
\begin{itemize}
 \item $(w_1-w_t)=4$ balls of weight $w_k=3$ from $\mc{B}_1$ to $\mc{B}_2$,
 \item $(w_1-w_k)=2$ balls of weight $w_t=1$ from $\mc{B}_2$ to $\mc{B}_1$,
 \item $(w_k-w_t)=2$ balls of weight $w_1=5$ from $\mc{B}_2$ to $\mc{B}_1$,
\end{itemize}
then we obtain the partition
\[
\setlength{\extrarowheight}{2pt}
\begin{array}{c|c|c|c|c|c|c}
a & 1 & 2  & 3 & 4 & 5 & 6 \\
 \hline 
\mc{B}_a & 5^2 3^2 1^2 & 5^1 3^4 0^1 & 5^3 1^2 0^1 & 5^3 1^1 0^2 & 5^3 1^1 0^2 & 1^6 \\ 
\end{array}
\]
that has $N=2$ balls of weight $w_1$ in $\mc{B}_1$, and has the same weight sequence $w(\mc{B}_{\bullet})$ as the original partition.
\end{example} 
 
\section{Hall Marriage and the permutation $R_{\bullet}$}\label{sec:Hall-Marriage} 

The goal of this section is to explain and make effective {\bf Step~4} of our outline from Section~\ref{sec:strategy-inductive-feasibility}. To that end, we prove the following.

\begin{theorem}\label{thm:existence-R}
 Fix a positive integer $r$, let $N=(n-1)\cdot n\cdot r$ and suppose that $d$ satisfies (\ref{eq:bound-d-manyw1}). Suppose that $\mf{BP}(d,\ul{C};\ul{w})$ is $1$-feasible and $W^{min}>C_1$, and consider a $1$-feasible partition $\mc{B}_{\bullet}$ with $w(\mc{B}_1)=W^{min}$ and $n_1(\mc{B}_1)\geq N$ (whose existence is guaranteed by Theorem~\ref{thm:many-w1}). There exists a permutation $R_{\bullet}$ of $\{1,2,\cdots,n\}$ with $R_1=1$, and with the property that $n_i(\mc{B}_{R_i})\geq r$ for $i=2,\cdots,n$.
\end{theorem}

\begin{proof}
We define a bipartite graph $G$ with vertex set $\mc{L}\sqcup \mc{R}$, where $\mc{L}=\mc{R}=\{2,\cdots,n\}$, where $i\in \mc{L}$ and $j\in \mc{R}$ are connected by an edge if and only if $n_i(\mc{B}_j)\geq r$. Our goal is to prove that $G$ admits a perfect matching: letting $R_1=1$ and $R_i=j$ when $i$ is matched to $j$ gives then the desired permutation $R_{\bullet}$.

Given any subset $\mc{S}\subset \mc{L}$, we consider the set of neighbors of elements in $\mc{S}$:
\[ \mc{N}(\mc{S}) = \{ j\in \mc{R} : (i,j)\mbox{ is an edge in }G\mbox{ for some }i\in\mc{S}\}.\]
By Hall's Marriage Theorem (see \cite{hall} or \cite[Theorem~10.4]{korte-vygen}) we need to check that for every subset $\mc{S}\subset \mc{L}$, we have $|\mc{N}(\mc{S})| \geq |\mc{S}|$. Suppose this isn't the case, and let $\mc{S}$ be such that
\[ m=|\mc{N}(\mc{S})| < k = |\mc{S}|.\]
We consider the set of weights indexed by $\mc{S}$,
\[ \mc{W}(\mc{S}) = \{ w_s : s\in \mc{S}\},\]
and let $M$ denote the number of balls in $\mc{B}$ whose weight belongs to $\mc{W}(\mc{S})$. We have two cases:

\noindent{\bf Case 1:} $w_1\in\mc{W}(\mc{S})$. Since each $w_i$, $i=1,\cdots,n$ appears $d$ times in $\mc{B}$, and since $1\not\in\mc{S}$ but $w_1\in\mc{W}(\mc{S})$, it follows that 
\[ M \geq k\cdot d + d.\]
If $j\neq 1$ and $j\not\in\mc{N}(\mc{S})$, we have that each $w_s\in\mc{W}(\mc{S})$ appears at most $(r-1)$ times in $\mc{B}_j$. It follows that
\[ M \leq (1+m)\cdot d + (n-m-1)\cdot k\cdot (r-1).\]
Combining the two inequalities above involving $M$, and using the fact that $k\geq 1+m$, we obtain
\[ d \leq (n-m+1)\cdot k\cdot (r-1) < n\cdot (n-1)\cdot r=N,\]
contradicting (\ref{eq:bound-d-manyw1}).

\noindent{\bf Case 2:} $w_1\not\in\mc{W}(\mc{S})$. We know that $n_1(\mc{B}_1)\geq N$, so at most $d-N$ balls in $\mc{B}_1$ can have weight in $\mc{W}(\mc{S})$. It follows that
\[ M \leq m\cdot d + (d-N) + (n-m-1)\cdot k\cdot (r-1).\]
Since $|S|=k$ it follows that $M\geq k\cdot d$, and using again that $k\geq 1+m$ we conclude that
\[ N \leq (n-m+1)\cdot k\cdot (r-1) < n\cdot (n-1)\cdot r=N,\]
which is again a contradiction.
\end{proof}

\begin{example}\label{ex:perm-R}
 Below is an example of a permutation $R_{\bullet}$ such that $\mc{B}_{R_i}$ contains many balls of weight $w_i$ for all $i$, where ``many" means in this case two balls. We take $\ul{w}=(5,5,3,1,1,0)$, $d=6$, $r=2$, as before, and the partition $\mc{B}_{\bullet}$ and permutation $R_{\bullet}$ given by:
\[
\setlength{\extrarowheight}{2pt}
\begin{array}{c|c|c|c|c|c|c}
a & 1 & 2  & 3 & 4 & 5 & 6 \\
 \hline 
\mc{B}_a & 5^2 3^2 1^2 & 5^1 3^4 0^1 & 5^3 1^2 0^1 & 5^3 1^1 0^2 & 5^3 1^1 0^2 & 1^6 \\ 
\hline
R_a & 1 & 4  & 2 & 3 & 6 & 5 \\
\end{array}
\]

\end{example}

\section{Shrinking gaps}\label{sec:shrink-gaps} 

The goal of this section is to formalize {\bf Step~5} of our outline from Section~\ref{sec:strategy-inductive-feasibility}. In Section~\ref{subsec:gap-sequence} we give an alternative interpretation of the constant $b(\ul{w})$, which leads to a set of inequalities that imply (\ref{eq:small-sum-gi}). We then introduce an algorithm in Section~\ref{subsec:min-gaps} that produces after a series of ball swaps a partition $\mc{B}'_{\bullet}$ either satisfying the said inequalities (in which case it is feasible by Lemma~\ref{lem:small-sum-gi}), or satisfying $w(\mc{B}'_1)<w(\mc{B}_1)$.

\subsection{The gap sequence of a tuple $\ul{w}$}\label{subsec:gap-sequence}

Consider $\ul{w}=(w_1,\cdots,w_n)$ as before, and define for $i=1,\cdots,n$:
\begin{itemize}
 \item The \defi{predecessor} $p(w_i)$ of $w_i$ to be $w_j$, where $j<i$ is the unique index such that $w_j>w_{j+1}=w_{j+2}=\cdots=w_i$. If $w_i=w_1$ then we make the convention that $p(w_i)=\infty$.
 \item The \defi{successor} $s(w_i)$ of $w_i$ to be $w_j$, where $j>i$ is the unique index such that $w_i=w_{i+1}=\cdots=w_{j-1}>w_j$. If $w_i=w_n$, then we make the convention that $s(w_i)=-\infty$.
\end{itemize}

We define the \defi{gap sequence} of $\ul{w}$ to be the list $g_{\bullet}(\ul{w}) = (g_2(\ul{w}),\cdots,g_n(\ul{w}))$ defined by letting
\[ g_i(\ul{w}) = \begin{cases}
p(w_i)-w_i & \mbox{ if }w_i\neq w_1,\\
w_1-s(w_1) & \mbox{ if }w_i=w_1.
\end{cases}
\]
Note that the formula above also makes sense for $i=1$, but this case is not relevant for our argument.

\begin{example}\label{ex:gap-w}
 Let $\ul{w}=(5,5,3,1,1,0)$. We have that $s(5)=3$, $p(3)=5$, $p(1)=3$, and $p(0)=1$. We get that the gap sequence of $\ul{w}$ is
 \[ g_{\bullet}(\ul{w}) = \bigl(g_2(\ul{w}),g_3(\ul{w}),g_4(\ul{w}),g_5(\ul{w}),g_6(\ul{w})\bigr) = (2,2,2,2,1).\]
\end{example}

\begin{lemma}\label{lem:bw=sum-gaps}
 Suppose that $\ul{w}$ has at least two distinct entries. We have that
 \[ b(\ul{w}) = \sum_{i=2}^n \bigl(g_i(\ul{w})-1\bigr).\]
\end{lemma}

\begin{proof}
 Write the conjugate partition to $\ll=\ul{w}$ as in Definition~\ref{def:b-lam}, $\ll'=(n^{a_0},h_1^{a_1},h_2^{a_2},\cdots,h_k^{a_k})$, with $a_0\geq 0$ and $a_1,\cdots,a_k>0$, $n>h_1>\cdots>h_k>0$. The assumption that $\ul{w}$ has at least two distinct entries is equivalent to the condition $k\neq 0$. We compute $g_i(\ul{w})$ for each $i=2,\cdots,n$.
 
 If $i>h_1$ then $w_i=a_0$, $p(w_i)=a_1+a_0$, so that $g_i(\ul{w})=a_1$. It follows that
 \begin{equation}\label{eq:sum-gap-h1+1}
 \sum_{i=h_1+1}^n \bigl(g_i(\ul{w})-1\bigr) = (n-h)\cdot(a_1-1).
 \end{equation}

For $j=1,\cdots,k-1$, if $h_j\geq i>h_{j+1}$ then $w_i=a_0+a_1+\cdots+a_j$, $p(w_i)=a_0+a_1+\cdots+a_{j+1}$, so that $g_i(\ul{w}) = a_{j+1}$ and
\begin{equation}\label{eq:sum-gap-hj-hj+1}
\sum_{i=h_{j+1}+1}^{h_j} \bigl(g_i(\ul{w})-1\bigr) = (h_j-h_{j+1})\cdot(a_{j+1}-1).
\end{equation}

If $h_k\geq i\geq 2$ then $w_i=w_1=a_0+\cdots+a_k$, $s(w_1) = a_0+\cdots+a_{k-1}$, so $g_i(\ul{w}) = a_k$ and
\begin{equation}\label{eq:sum-gap-2-hk}
\sum_{i=2}^{h_k} \bigl(g_i(\ul{w})-1\bigr) = (h_k-1)\cdot(a_k-1).
\end{equation}

Summing together (\ref{eq:sum-gap-h1+1}--\ref{eq:sum-gap-2-hk}) and comparing with Definition~\ref{def:b-lam} we obtain the desired conclusion.
\end{proof}

\begin{corollary}\label{cor:small-gi}
 Suppose that (\ref{eq:size-C-feasibility}) holds and that $\mc{B}_{\bullet}$ is $1$-feasible. If
 \begin{equation}\label{eq:small-gi}
  g_i(\mc{B}_{\bullet}) \leq g_i(\ul{w}) - 1\mbox{ for all }i=2,\cdots,n
 \end{equation}
 then $\mc{B}_{\bullet}$ is feasible.
\end{corollary}

\begin{proof}
 Summing together the inequalities (\ref{eq:small-gi}) for $i=2,\cdots,n$ we obtain using Lemma~\ref{lem:bw=sum-gaps} the inequality (\ref{eq:small-sum-gi}), and conclude using Lemma~\ref{lem:small-sum-gi} that $\mc{B}_{\bullet}$ is feasible.
\end{proof}

\subsection{Shrinking the gaps through ball swaps}\label{subsec:min-gaps}

We let $r=2\cdot n\cdot w_1$, $N=(n-1)\cdot n\cdot r$ and suppose that $d$ satisfies (\ref{eq:bound-d-manyw1}), that is
\begin{equation}\label{eq:bounds-r-d}
d \geq 2\cdot n^3\cdot(n-1)\cdot w_1 + n\cdot(n-1)\cdot w_1 \cdot |\ul{w}|.
\end{equation}
We assume that $\mc{B}_{\bullet}$ is $1$-feasible, and $w(\mc{B}_1)=W^{min}>C_1$ (so that $g_1(\mc{B}_{\bullet})<0$ and $\mc{B}_{\bullet}$ is not feasible). Using Theorem~\ref{thm:existence-R} (and the fact that $N\geq r$), we can find a permutation $R$ of the set $\{1,\cdots,n\}$, with
\[R_1=1,\mbox{ and }n_i(\mc{B}_{R_i})\geq r \mbox{ for }i=1,\cdots,n.\]
We will also assume that $\ul{w}$ has at least two distinct entries, and let $w^{2nd} = s(w_1)$ denote the second largest weight in the sequence $\ul{w}$. We construct a series of exchanges that will produce out of the $1$-feasible partition $\mc{B}_{\bullet}$ a new one $\mc{B}_{\bullet}'$ with $w(\mc{B}'_1) < w(\mc{B}_1)$, contradicting the minimality of $w(\mc{B}_1)$. Using Lemma~\ref{lem:swap-B1-w1}, we may further assume that $g_i(\mc{B}_{\bullet})<w_1$ for all $i=2,\cdots,n$. We show the following.

\begin{theorem}\label{thm:shrinking-gaps}
 If $\ul{w}$, $\mc{B}_{\bullet}$, $R_{\bullet}$ are as above, then the output of Algorithm~\ref{alg:gaps} below is a $1$-feasible partition $\mc{B}_{\bullet}'$ with $w(\mc{B}'_1) < w(\mc{B}_1)$.
\end{theorem}

\noindent Note that this theorem leads to the contradiction in {\bf Step 5} of the outline from Section~\ref{sec:strategy-inductive-feasibility}. Note also that
\[ 2\cdot n^3\cdot(n-1)\cdot w_1 + n\cdot(n-1)\cdot w_1 \cdot |\ul{w}| \leq 2\cdot n^4\cdot w_1 + n^2\cdot w_1\cdot n\cdot w_1 = n^3\cdot w_1\cdot(2\cdot n + w_1),\]
which gives the effective bound for $d^1_{\ul{w}}$ in Theorem~\ref{thm:inductive-feasibility}. The proof of Theorem~\ref{thm:shrinking-gaps} will occupy the rest of the section: we break it up into simple steps as follows (the reader may wish to go through Examples~\ref{ex:alg-gaps} and~\ref{ex:lines-14-17} before getting into more details).


\begin{algorithm}
  \KwIn{$\ul{w}=(w_1,\cdots,w_n)$, a non-increasing tuple of weights, at least two distinct\newline 
  $\mc{B}_{\bullet}$, a $1$-feasible partition which is not feasible, with $g_i(\mc{B}_{\bullet})<w_1$ for $i=2,\cdots,n$\newline 
  $R_{\bullet}$, a permutation of $\{1,\cdots,n\}$ such that $R_1=1$ and $n_i(\mc{B}_{R_i})\geq r$ for all $i=1,\cdots,n$}
  \KwOut{$\mc{B}'_{\bullet}$, a $1$-feasible partition with $w(\mc{B}'_1) < w(\mc{B}_1)$}
  $i \leftarrow n$\;
  \While{$w_i<w^{2nd}$}{
  	\While{$g_{R_i}(\mc{B}_{\bullet}) \geq g_i(\ul{w})$}{
  	choose $j<i$ such that $w_j=p(w_i)$\;
	swap a ball of weight $w_i$ in $\mc{B}_{R_i}$ with a ball of weight $w_j$ in $\mc{B}_{R_j}$\;
	\If{$g_{R_j}(\mc{B}_{\bullet})\geq w_1$}{
		swap a ball of weight $w_1$ in $\mc{B}_1$ with a ball of weight $w_i$ in $\mc{B}_{R_j}$\;
		go to line 22
		}
	}
	$i \leftarrow i-1$\;
  }
  $i \leftarrow 2$\;
 \While{$w_i \geq w^{2nd}$}{
 	\While{$g_{R_i}(\mc{B}_{\bullet})\geq g_i(\ul{w})$ }{
		\eIf{$\mc{B}_{R_i}$ contains no ball of weight $w^{2nd}$}{
		choose $j>i$ such that $w_j=w^{2nd}$\;
		choose $l>j$ such that $w_l<w^{2nd}$ and $\mc{B}_{R_i}$ contains a ball of weight $w_l$\;
		$t\leftarrow\displaystyle\left\lfloor\frac{w_j-w_l}{g_i(\ul{w})}\right\rfloor $ \;
		swap $(t+1)$ balls of weight $w_j$ from $\mc{B}_{R_j}$ with $t$ balls of weight $w_1$ and one ball of weight $w_l$ from $\mc{B}_{R_i}$
		}
 		{swap a ball of weight $w_1$ from $\mc{B}_1$ with one of weight $w^{2nd}$ from $\mc{B}_{R_i}$\;
		go to line 22
		}
	}	
	$i \leftarrow i+1$\;
 }
  \Return{$\mc{B}_{\bullet}$}
  \caption{Shrinking gaps algorithm}
  \label{alg:gaps}
\end{algorithm}


\begin{lemma}\label{lem:swaps-5}
 For a fixed index $i$, the swap in line $5$ of the algorithm is repeated fewer than $w_1$ times. Moreover, the partition $\mc{B}_{\bullet}$ remains $1$-feasible after each swap.
\end{lemma}

\begin{proof}
 Every swap in line $5$ of the algorithm increases $w(\mc{B}_{R_i})$ and decreases $w(\mc{B}_{R_j})$ by 
 \[1\leq w_j-w_i=g_i(\ul{w}).\]
 Since $g_{R_i}(\mc{B}_{\bullet})<w_1$, this occurs at most $(w_1-1)$ times. Since $g_{R_i}(\mc{B}_{\bullet})\geq g_i(\ul{w})$ is satisfied at every swap in line $5$, and since $g_{R_i}(\mc{B}_{\bullet})$ is decreased by $g_i(\ul{w})$, it follows that $g_{R_i}(\mc{B}_{\bullet})$ stays non-negative and therefore $\mc{B}_{\bullet}$ remains $1$-feasible (the only other gap that changes is $g_{R_j}(\mc{B}_{\bullet})$, but it gets larger and thus stays non-negative).
\end{proof}

\begin{lemma}\label{lem:wj-in-line-5}
 The swap in line $5$ of the algorithm occurs fewer than $n\cdot w_1$ times. 
\end{lemma}

\begin{proof}
 Using Lemma~\ref{lem:swaps-5}, the conclusion follows from the fact that there are fewer than $n$ indices $i$ with $w_i<w^{2nd}$, which is clear (in fact, there are at most $n-2$ such indices).
\end{proof}

Since every swap in line $5$ of the algorithm decreases by at most one the number of balls of a given weight in any given bin, it follows from Lemmas~\ref{lem:swaps-5},~\ref{lem:wj-in-line-5} that we have the following.

\begin{corollary}\label{cor:r-large-line-5}
 Since $r\geq n\cdot w_1$, at every run through line $5$ of the algorithm we are guaranteed to have at least one ball of weight $w_i$ in $\mc{B}_{R_i}$, and one of weight $w_j$ in $\mc{B}_{R_j}$, so the swap can be performed.
\end{corollary}

\begin{lemma}\label{lem:run-6-8}
 Every run through lines 6--8 of the algorithm either decreases $w(\mc{B}_1)$ or guarantees that the inequalities $g_a(\mc{B}_{\bullet})<w_1$ for $a=2,\cdots,n$ remain valid.
\end{lemma}

\begin{proof}
 If the inequality in line $6$ is satisfied, then line $7$ produces a partition where $w(\mc{B}_1)$ is decreased by $w_1-w_i\geq w_1-w^{2nd}>0$. If the inequality in line $6$ fails, then $g_{R_j}(\mc{B}_{\bullet})<w_1$, so the condition $g_a(\mc{B}_{\bullet})<w_1$ remains valid for $a\geq 2$, since for $a\neq R_j$ no $g_a(\mc{B}_{\bullet})$ is increased by the swap in line $5$ of the algorithm.
\end{proof}

\begin{remark}\label{rem:lines-1-9}
After running the first $9$ lines of the algorithm, the partition $\mc{B}_{\bullet}$ has the following properties:
\begin{enumerate}
 \item For every $i$ with $w_i<w^{2nd}$ we have that $g_{R_i}(\mc{B}_{\bullet})<g_i(\ul{w})$.
 \item For every $i$ we have that $\mc{B}_{R_i}$ contains more than $r-n\cdot w_1=n\cdot w_1$ balls of weight $w_i$.
\end{enumerate}
Indeed, conclusion (1) is just a reformulation of the failure of the inequality in line 3, while conclusion (2) follows from Lemma~\ref{lem:wj-in-line-5}.
\end{remark}

Our next goal is to show that the second part of the algorithm yields conclusion (1) in Remark~\ref{rem:lines-1-9} also for each $i\geq 2$ for which $w_i\geq w^{2nd}$ (or it results in a partition with a lower $w(\mc{B}_1)$).

\begin{lemma}\label{lem:line13}
 For a fixed $i$, the condition in line 13 is satisfied at most once.
\end{lemma}

\begin{proof}
 The swap in line 17 places $t+1\geq 1$ balls of weight $w_j=w^{2nd}$ into $\mc{B}_{R_i}$, so the condition in line 13 can't be satisfied a second time for the same value of $i$. 
\end{proof}

\begin{lemma}\label{lem:swap17}
 The swap in line 17 occurs fewer than $n$ times.
\end{lemma}

\begin{proof}
 Since there are fewer than $n$ values of $i\geq 2$ for which $w_i\geq w^{2nd}$, and since for each such value the condition in line 13 is satisfied at most once, the conclusion follows.
\end{proof}

\begin{lemma}\label{lem:t-line16}
 The value of $t$ in line 16 is smaller than $w_1$.
\end{lemma}

\begin{proof}
 We have that $t\leq w_j-w_l < w_1$, since $w_j=w^{2nd}<w_1$.
\end{proof}

\begin{lemma}\label{lem:lines-13-14-15}
 The condition in line 13 can only be satisfied when $w_i=w_1$. Moreover, when it is satisfied we have that indices $j,l$ as in lines $14$ and $15$ exist.
\end{lemma}

\begin{proof}
 When the loop in line 11 is initiated, we know by Remark~\ref{rem:lines-1-9} that each $\mc{B}_{R_a}$ contains more than $r-n\cdot w_1=n\cdot w_1$ balls of weight $w_a$. The swap in line $17$ occurs fewer than $n$ times by Lemma~\ref{lem:swap17}, and each time it removes at most $w_1\geq t+1$ balls of weight $w_a$ from $\mc{B}_{R_a}$ by Lemma~\ref{lem:t-line16}, so at any point we have that each $\mc{B}_{R_a}$ contains at least $w_1$ balls of weight $w_a$.
 
 To prove the first assertion, note that if $w_i\geq w^{2nd}$ then either $w_i=w_1$ or $w_i=w^{2nd}$. Since $\mc{B}_{R_i}$ contains balls of weight $w_i$, it follows that for $w_i=w^{2nd}$ the condition in line $13$ must fail. 
 
 We now assume that the condition in line $13$ is satisfied, and in particular $w_i=w_1$. We can choose $j>i$ with $w_j=w^{2nd}$ since $\ul{w}$ is non-decreasing. If the index $l$ in line $15$ did not exist, then $\mc{B}_{R_i}$ would have to consist of $d$ balls of weight $w_1$, contradicting the conclusion of Lemma~\ref{lem:Bi-not-only-w1}.
\end{proof}

\begin{lemma}\label{lem:est-t-lines-16-17}
 The swap in line $17$ can always be performed, and the resulting $\mc{B}_{\bullet}$ stays $1$-feasible.
\end{lemma}

\begin{proof}
 As explained in the proof of Lemma~\ref{lem:lines-13-14-15}, we have at every point that each $\mc{B}_{R_a}$ contains at least $w_1$ balls of weight $w_a$. Since $t+1\leq w_1$, we get that $\mc{B}_{R_j}$ contains $(t+1)$ balls of weight $w_j$. By Lemma~\ref{lem:lines-13-14-15} we know that $w_i=w_1$, so $\mc{B}_{R_i}$ contains $t$ balls of weight $w_1$. Since $\mc{B}_{R_i}$ also contains a ball of weight $w_l$ by Lemma~\ref{lem:lines-13-14-15}, the swap can be performed.

To check $1$-feasibility, note that the swap in line $17$ adds to $g_{R_j}(\mc{B}_{\bullet})$ (and subtracts from $g_{R_i}(\mc{B}_{\bullet})$)
\[(t+1)\cdot w_j - t\cdot w_i -w_l = (w_j-w_l) - t\cdot(w_i-w_j).\]
Noting that $w_i-w_j=g_i(\ul{w})$, we see by the choice of $t$ that the quantity above is a non-negative integer $\leq g_i(\ul{w})$. Since $g_{R_i}(\mc{B}_{\bullet})\geq g_i(\ul{w})$ before the swap, the value of $g_{R_i}(\mc{B}_{\bullet})$ remains non-negative after the swap. Since this is the only gap that is decreased, $1$-feasibility is preserved.
\end{proof}

\begin{lemma}\label{lem:swap-18}
 The swap in line $19$ is possible and the resulting $\mc{B}_{\bullet}$ stays $1$-feasible.
\end{lemma}

\begin{proof}
 Since $r>0$, $\mc{B}_1$ contains at least one ball of weight $w_1$. If the condition in line 13 fails, then $\mc{B}_{R_i}$ contains a ball of weight $w^{2nd}$, so the swap can be performed. Since the only gap that is decreased is $g_{R_i}(\mc{B}_{\bullet})$, and the decrease is by $w_1-w^{2nd}=g_i(\ul{w})$, the conclusion follows using the inequality in line 12.
\end{proof}

\begin{proof}[Proof of Theorem~\ref{thm:shrinking-gaps}]
 The partition $\mc{B}'_{\bullet}$ returned by the algorithm occurs in one of the following ways:
 \begin{itemize}
  \item After the swap in line 7: since $w_1>w_i$, we have that $w(\mc{B}'_1)<w(\mc{B}_1)$.
  \item After the swap in line 19: since $w_1>w^{2nd}$, we have that $w(\mc{B}'_1)<w(\mc{B}_1)$.
  \item After the completion of the loop in lines 11--21: this implies that $g_{R_i}(\mc{B}'_{\bullet})<g_i(\ul{w})$ for all $i\geq 2$ for which $w_i\geq w^{2nd}$. Combining this with Remark~\ref{rem:lines-1-9}(1), we conclude that the inequalities (\ref{eq:small-gi}) hold for $\mc{B}'_{\bullet}$, so by Corollary~\ref{cor:small-gi}, $\mc{B}'_{\bullet}$ is feasible. This means that $w(\mc{B}'_1)\leq C_1<w(\mc{B}_1)$.\qedhere
 \end{itemize}
\end{proof}

\begin{example}\label{ex:alg-gaps}
 To indicate how Algorithm~\ref{alg:gaps} works, we consider the following example. We take $n=6$, $\ul{w}=(5,5,3,1,1,0)$, $d=6$, $r=2$, $\ul{C}=(17^5,8)$, and the partition $\mc{B}_{\bullet}$ and permutation $R_{\bullet}$ given below:
\[
\setlength{\extrarowheight}{2pt}
\begin{array}{c|c|c|c|c|c|c}
a & 1 & 2  & 3 & 4 & 5 & 6 \\
 \hline 
\mc{B}_a & 5^2 3^2 1^2 & 5^1 3^4 0^1 & 5^3 1^2 0^1 & 5^3 1^1 0^2 & 5^3 1^1 0^2 & 1^6 \\ 
\hline
R_a & 1 & 4  & 2 & 3 & 6 & 5 \\
 \hline 
w(\mc{B}_a) & 18 & 17 & 17 & 16 & 16 & 6 \\
\hline
g_{R_a}(\mc{B}_\bullet) & -1 & 1 & 0 & 0 & 2 & 1 \\
\hline
g_a(\ul{w}) &  & 2 & 2 & 2 & 2 & 1 \\
\end{array}
\]
The table below indicates how the partition $\mc{B}_{\bullet}$ changes as we run through the algorithm; a blank space means the corresponding part remains unchanged. The double line separates the first half of the algorithm (lines 1--9, where $i\in\{6,5,4\}$) from the second half (lines 10--20, where $i\in\{2,3\}$). In the leftmost column we indicate the line of the algorithm and the value of the relevant parameters where the exchange modifying $\mc{B}_{\bullet}$ occurs.
\[
\setlength{\extrarowheight}{2pt}
\begin{array}{c|c|c|c|c|c|c|c}
& a & 1 & 2  & 3 & 4 & 5 & 6 \\
 \hline 
\mbox{initial }\mc{B}_{\bullet} & \mc{B}_a & 5^2 3^2 1^2 & 5^1 3^4 0^1 & 5^3 1^2 0^1 & 5^3 1^1 0^2 & 5^3 1^1 0^2 & 1^6 \\ 
\hline
i = 6,\ j=5,\mbox{ line }5 & \mc{B}_a &  &  &  &  & 5^3 1^2 0^1 & 1^5 0^1 \\ 
\hline
i = 5,\ j=3,\mbox{ line }5 & \mc{B}_a &  & 5^1 3^3 1^1 0^1 &  &  &  & 3^1 1^4 0^1 \\ 
\hline
\hline
i = 3,\mbox{ line }19 & \mc{B}_a & 5^1 3^3 1^2 & 5^2 3^2 1^1 0^1 & &  & & \\ 
\end{array}
\]
Note that the resulting partition is the one considered in Example~\ref{ex:feasible-BP} and is feasible.
\end{example}

In the previous example the condition in line 13 was never satisfied, so lines 14--17 were never executed. To illustrate their contribution to the algorithm we consider the following.

\begin{example}\label{ex:lines-14-17}
We take $n=5$, $\ul{w}=(5,5,5,1,0)$, $d=5$, $r=3$, $\ul{C}=(24^3,15,2)$, and the partition $\mc{B}_{\bullet}$ and permutation $R_{\bullet}$ given below:
\[
\setlength{\extrarowheight}{2pt}
\begin{array}{c|c|c|c|c|c}
a & 1 & 2  & 3 & 4 & 5  \\
 \hline 
\mc{B}_a & 5^5 & 5^4 0^1 & 5^4 0^1 & 5^2 1^3 & 1^2 0^3 \\ 
\hline
R_a & 1 & 2  & 3 & 4 & 5 \\
 \hline 
w(\mc{B}_a) & 25 & 20 & 20 & 13 & 2 \\
\hline
g_{R_a}(\mc{B}_\bullet) & -1 & 4 & 4 & 2 & 0 \\
\hline
g_a(\ul{w}) &  & 4 & 4 & 4 & 1 \\
\end{array}
\]
The first half of the algorithm (lines 1--9) do not affect $\mc{B}_{\bullet}$. Using the same conventions as in Example~\ref{ex:alg-gaps}, we record the evolution of $\mc{B}_{\bullet}$ in the following table.
\[
\setlength{\extrarowheight}{2pt}
\begin{array}{c|c|c|c|c|c|c}
& a & 1 & 2  & 3 & 4 & 5 \\
 \hline 
\mbox{initial }\mc{B}_{\bullet} & \mc{B}_a & 5^5 & 5^4 0^1 & 5^4 0^1 & 5^2 1^3 & 1^2 0^3 \\ 
\hline
\hline
i = 2,\ j=4,\ l=5,\ t=0,\mbox{ line }17 & \mc{B}_a &  & 5^4 1^1 &   & 5^2 1^2 0^1 &  \\ 
\hline
i = 3,\ j=4,\ l=5,\ t=0,\mbox{ line }17 & \mc{B}_a &  &  &  5^4 1^1 & 5^2 1^1 0^2 &  \\ 
\hline
i = 4,\mbox{ line }19 & \mc{B}_a & 5^4 1^1 &  &  & 5^3 0^2 &  \\ 
\end{array}
\]
The resulting partition is therefore
\[
\setlength{\extrarowheight}{2pt}
\begin{array}{c|c|c|c|c|c}
a & 1 & 2  & 3 & 4 & 5 \\
 \hline 
\mc{B}'_a & 5^4 1^1 & 5^4 1^1 & 5^4 1^1 & 5^3 0^2 & 1^2 0^3 \\ 
\end{array}
\]
It satisfies $w(\mc{B}'_1)=21<25=w(\mc{B}_1)$, and in fact it is feasible.
\end{example}

\section{Optimality of the constant $b(\ul{w})$}\label{sec:optimality}

The goal of this section is to prove Theorem~\ref{thm:optimal-bw}. Example~\ref{ex:optimal-1-feasible} at the end of the section may be helpful in following the notation and details of the proof. We let $\ll=\ul{w}$ and write $\ll'=(n^{a_0},h_1^{a_1},h_2^{a_2},\cdots,h_k^{a_k})$, noting that $k>0$ as in the proof of Lemma~\ref{lem:bw=sum-gaps}. We set $h_0=n$ and recall that $w^{2nd}=s(w_1)$ denotes the second largest weight in $\ul{w}$. We define a sequence of capacities $\ul{C}^{\circ}$ as follows:
\begin{itemize}
 \item $C^{\circ}_1 = C^{\circ}_2 = \cdots = C^{\circ}_{h_k} = d\cdot w_1-1$.
 \item $C^{\circ}_{h_k+1} = (h_k-1)\cdot w_1 + (d-h_k+1)\cdot w^{2nd} +  (a_k-1)$.
 \item $C^{\circ}_j = d\cdot w^{2nd} + (a_k-1) = d w_j + (a_k-1)$ for $h_k+1 < j \leq h_{k-1}$.
 \item $C^{\circ}_j = d\cdot w_j + (a_{k-t}-1)$ for $t=1,\cdots,k-1$, and $h_{k-t} < j \leq h_{k-1-t}$.
\end{itemize}
We first check that condition (\ref{eq:opt-1}) is satisfied. By Definition~\ref{def:b-lam} and the proof of Lemma~\ref{lem:bw=sum-gaps}, we have
\[ g_j(\ul{w}) = \begin{cases}
 a_k & 2\leq j\leq h_{k-1}; \\
 a_{k-t} & h_{k-t} < j \leq h_{k-1-t},\ 1\leq t\leq k-1.
\end{cases}
\]
It follows from the definition of $\ul{C}^{\circ}$ that if we let $\Delta^{\circ}_j=C^{\circ}_j - d\cdot w_j$ then 
\begin{equation}\label{eq:formulas-Ccirc}
\Delta^{\circ}_j = \begin{cases}
 -1 & 1\leq j\leq h_k; \\
 (h_k-1)\cdot a_k + (a_k-1) = (h_k-1)\cdot a_k + (g_j(\ul{w})-1) &  j= h_k+1; \\
 a_k-1 = g_j(\ul{w})-1&  h_k+1 < j \leq h_{k-1}; \\
 a_{k-t}-1 = g_j(\ul{w})-1 & h_{k-t} < j \leq h_{k-1-t},\ 1\leq t\leq k-1.
\end{cases}
\end{equation}
Summing over $j=2,\cdots,n$ and using Lemma~\ref{lem:bw=sum-gaps} we conclude that
\[
\begin{aligned}
 \sum_{j=1}^n (C^{\circ}_j - d\cdot w_j) &=  -h_k + (h_k-1)\cdot a_k + \sum_{j=h_k+1}^n(g_j(\ul{w})-1) \\
 &= -1 + (h_k-1)\cdot (a_k-1) + \sum_{j=h_k+1}^n(g_j(\ul{w})-1) \\
 &= -1+\sum_{j=2}^n(g_j(\ul{w})-1) = b(\ul{w})-1.
\end{aligned}
\]
We next check that $\ul{C}^{\circ}$ is non-decreasing. Since $d\geq n\geq h_k$, we have that
\[ C^{\circ}_{h_k} - C^{\circ}_{h_k+1} = (d-h_k+1)\cdot(w_1-w^{2nd}) - a_k = (d-h_k)\cdot a_k \geq 0.\]
Similarly, we get that $C^{\circ}_{h_k+1} \geq d\cdot w^{2nd} + (a_k-1)$. When $j=h_{k-t}+1$ for $1\leq t\leq k-1$, we have that the predecessor of $w_j$ is $p(w_j)=w_{j-1}$, so $b_j(\ul{w}) = w_{j-1}-w_j = a_{k-t}$, and thus
\[ C^{\circ}_j = d\cdot w_j + (a_{k-t}-1) = d\cdot w_j + (w_{j-1}-w_j-1) < d\cdot w_{j-1} \leq C^{\circ}_{j-1}.\]

We next show that $\mf{BP}(d,\ul{C}^{\circ};\ul{w})$ is $1$-feasible. We consider the partition $\mc{B}^{\circ}_{\bullet}$ defined by
\begin{itemize}
 \item $\mc{B}^{\circ}_1=\{w_1^d\}$.
 \item $\mc{B}^{\circ}_2 = \cdots = \mc{B}^{\circ}_{h_k} = \{w_1^{d-1},w^{2nd}\}$.
 \item $\mc{B}^{\circ}_{h_k+1} = \{w_1^{h_k-1},(w^{2nd})^{d-h_k+1}\}$.
 \item $\mc{B}^{\circ}_j = \{w_j^d\}$ for $j\geq h_k+2$.
\end{itemize}
Note that we are using the fact that $d\geq h_k-1$ in order for the definition of $\mc{B}^{\circ}_{h_k+1}$ to make sense. The earlier calculations show that $g_1(\mc{B}^{\circ}_{\bullet})=-1$, and $g_j(\mc{B}^{\circ}_{\bullet})=g_j(\ul{w})\geq 0$ for $j\geq 2$, so $\mc{B}^{\circ}_{\bullet}$ is $1$-feasible. 

To finish the proof, we need to verify that there exists no feasible partition $\mc{B}_{\bullet}$. Suppose by contradiction that there is one such $\mc{B}_{\bullet}$: for $j\geq h_k+2$, we prove by descending induction on $j$ that 
\begin{equation}\label{eq:union-Bj-n}
\mc{B}_j \sqcup \mc{B}_{j+1}\sqcup \cdots \sqcup \mc{B}_n = \{w_j^d,w_{j+1}^d,\cdots,w_n^d\}.
\end{equation}
When $j>n$ there is nothing to prove. Suppose that (\ref{eq:union-Bj-n}) holds for some $j>h_k+2$, so that $\mc{B}_1\sqcup\cdots\sqcup\mc{B}_{j-1}$ contains no ball of weight smaller than $w_{j-1}$. If $\mc{B}_{j-1}$ contains a ball of weight larger than $w_{j-1}$ then that weight is at least $p(w_{j-1}) = w_{j-1} + g_{j-1}(\ul{w})$. It follows that
\[ w(\mc{B}_{j-1}) \geq (d-1)\cdot w_{j-1} + (w_{j-1} + g_{j-1}(\ul{w})) = C^{\circ}_{j-1} + 1,\]
which contradicts the fact that $\mc{B}_{\bullet}$ is feasible. We conclude that $\mc{B}_{j-1} = \{w_{j-1}^d\}$, proving the induction step.

Since $w_1=\cdots=w_{h_k}$ and $w_{h_k+1}=w^{2nd}$, it follows that
\[ \mc{B}_1 \sqcup \cdots \sqcup \mc{B}_{h_k+1} = \{ w_1^{d\cdot h_k}, (w^{2nd})^d\}.\]
Since $C^{\circ}_1=\cdots=C^{\circ}_{h_k}<d\cdot w_1$, it follows that each of $\mc{B}_1,\cdots,\mc{B}_{h_k}$ contains at most $(d-1)$ balls of weight~$w_1$. This implies that $\mc{B}_{h_k+1}$ must contain at least $h_k$ balls of weight $w_1$, so
\[ (h_k-1)\cdot w_1 + (d-h_k+1)\cdot w^{2nd} +  (a_k-1) = C^{\circ}_{h_k+1} \geq w(\mc{B}_{h_k+1}) \geq h_k\cdot w_1 + (d-h_k)\cdot w^{2nd},\]
which implies that $a_k-1 \geq w_1-w^{2nd} = a_k$, a contradiction. This proves that $\mf{BP}(d,\ul{C};\ul{w})$ is not feasible, as desired.

\begin{example}\label{ex:optimal-1-feasible}
 If $\ll = \ul{w}=(5,5,3,1,1,0)$ then $\ll' = (5^1,3^2,2^2)$, so $k=3$, $h_1=5$, $h_2=3$, $h_3=2$, $a_1=1$, $a_2=a_3=2$. If we take $d=6$ then the partition $\mc{B}^{\circ}_{\bullet}$ and the bin capacity sequence are as follows.
\[
\setlength{\extrarowheight}{2pt}
\begin{array}{c|c|c|c|c|c|c}
a & 1 & 2  & 3 & 4 & 5 & 6 \\
 \hline 
\mc{B}^{\circ}_a & 5^6 & 5^5 3^1 & 5^1 3^5 & 1^6 & 1^6 & 0^6 \\ 
 \hline 
C^{\circ}_a & 29 & 29 & 21 & 7 & 7 & 0 \\ 
 \hline 
w(\mc{B}^{\circ}_a) & 30 & 28 & 20 & 6 & 6 & 0 \\ 
\end{array}
\]
Note that $\mc{B}^{\circ}_{\bullet}$ is $1$-feasible but not feasible, and that $g_i(\mc{B}^{\circ}_{\bullet}) = g_i(\ul{w})-1$ for all $i\geq 2$. Recall that $b(\ul{w})=4$, and note that $C^{\circ}_1+\cdots+C^{\circ}_n=93$, and $d\cdot(w_1+\cdots+w_n)=90$, just as in Example~\ref{ex:optimality-main-thm}. The key difference is that in this case no feasible solution exists!
\end{example}

\section{The proof of the feasibility criterion}\label{sec:feasibility}

The goal of this section is to prove Theorem~\ref{thm:sufficient-condition}. The fact that conditions (\ref{eq:sufficient-cond-C-w}) are sufficient for feasibility follows inductively from Theorem~\ref{thm:inductive-feasibility} and is explained in Section~\ref{subsec:sufficiency}. The optimality of the conditions however is not a formal consequence of Theorem~\ref{thm:optimal-bw}, and we discuss this issue in Section~\ref{subsec:main-thm-optimal}. We start with the following useful observation.

\begin{lemma}\label{lem:feasibility-truncations}
 Suppose that $i\geq 1$, $j\geq 0$. We have that $\mf{BP}(d,\ul{C};\ul{w})$ is $(i+j-1)$-feasible if and only if $\mf{BP}(d,\ul{C}^{\geq i};\ul{w}^{\geq i})$ is $j$-feasible.
\end{lemma}

\begin{proof} ``$\Leftarrow$": consider a $j$-feasible solution of $\mf{BP}(d,\ul{C}^{\geq i};\ul{w}^{\geq i})$, so that
\begin{equation}\label{eq:Bi-sqcup-Bn}
 \mc{B}_i \sqcup \mc{B}_{i+1}\sqcup \cdots \sqcup \mc{B}_n = \{w_i^d,\cdots,w_n^d\}.
\end{equation}
and $w(\mc{B}_t) \leq C_t$ for $t\geq i+j$. If we define $\mc{B}_t = \{w_t^d\}$ for $t<i$ then $\mc{B}_1\sqcup\cdots\sqcup\mc{B}_n$ is an $(i+j-1)$-feasible solution of $\mf{BP}(d,\ul{C};\ul{w})$, proving the implication.

\noindent``$\Rightarrow$": Let $\mc{B}_1\sqcup\cdots\sqcup\mc{B}_n$ be an $(i+j-1)$-feasible solution of $\mf{BP}(d,\ul{C};\ul{w})$, so that $w(\mc{B}_t) \leq C_t$ for $t\geq i+j$. If (\ref{eq:Bi-sqcup-Bn}) holds then it follows that it provides a $j$-feasible solution of $\mf{BP}(d,\ul{C}^{\geq i};\ul{w}^{\geq i})$, as desired. If (\ref{eq:Bi-sqcup-Bn}) does not hold then we perform a sequence of ball swaps that preserve the inequalities $w(\mc{B}_t) \leq C_t$ for $t\geq i+j$ and leads to a partition $\mc{B}_{\bullet}$ satisfying (\ref{eq:Bi-sqcup-Bn}), as follows. 

We choose $d\cdot(i-1)$ balls with $d$ of weight $w_j$ for each $j=1,\cdots,i-1$, and designate them as \emph{large}, and we designate the remaining $d\cdot(n-i+1)$ balls as \emph{small}. We write
\begin{equation}\label{eq:def-B-truncs}
 \mc{B}^{\leq i-1} = \mc{B}_1\sqcup\cdots\sqcup\mc{B}_{i-1}\mbox{ and }\mc{B}^{\geq i} = \mc{B}_i \sqcup \cdots \sqcup \mc{B}_n,
\end{equation}
and note that if $\mc{B}^{\geq i}$ consists entirely of small balls then (\ref{eq:Bi-sqcup-Bn}) holds. Note that if we swap a small ball from $\mc{B}^{\leq i-1}$ with a large ball from $\mc{B}^{\geq i}$, then the value of $w(\mc{B}_t)$ can only go down for $t\geq i+j$. It follows that after swapping each small ball from $\mc{B}^{\leq i-1}$ with a corresponding large ball from $\mc{B}^{\geq i}$, we get (\ref{eq:Bi-sqcup-Bn}), as desired.
\end{proof}

\subsection{Sufficiency}\label{subsec:sufficiency}

In this section we assume that (\ref{eq:sufficient-cond-C-w}) holds and show that $\mf{BP}(d,\ul{C};\ul{w})$ is feasible for $d\geq d_{\ul{w}}$ (note that our choice implies $d_{\ul{w}}=d^1_{\ul{w}}\geq d^1_{\ul{w}^{\geq i}}$ for all $i$). For $i\geq 0$, we prove by descending induction, starting with $i=n$, that $\mf{BP}(d,\ul{C};\ul{w})$ is $i$-feasible. When $i=n$ we have that $\ul{w}^{\geq n} = (w_n)$ is a singleton, therefore $b(\ul{w}^{\geq n})=0$, and (\ref{eq:sufficient-cond-C-w}) implies $C_n\geq d\cdot w_n$. If we let $\mc{B}_n=\{w_n^d\}$ and distribute $d$ balls to each of the bins $\mc{B}_1,\cdots,\mc{B}_{n-1}$ in an arbitrary fashion, then we obtain an $n$-feasible solution of $\mf{BP}(d,\ul{C};\ul{w})$.

Suppose now that $1\leq i\leq n$ and that $\mf{BP}(d,\ul{C};\ul{w})$ is $i$-feasible. We have that $\mf{BP}(d,\ul{C}^{\geq i};\ul{w}^{\geq i})$ is $1$-feasible by taking $j=1$ in Lemma~\ref{lem:feasibility-truncations}, and by Theorem~\ref{thm:inductive-feasibility} and the hypothesis
\[ C_i + \cdots + C_n \geq d\cdot(w_i+\cdots+w_n) + b(\ul{w}^{\geq i})\]
we conclude that $\mf{BP}(d,\ul{C}^{\geq i};\ul{w}^{\geq i})$ is also feasible (that is, $0$-feasible). Applying Lemma~\ref{lem:feasibility-truncations} with $j=0$ we conclude that $\mf{BP}(d,\ul{C};\ul{w})$ is $(i-1)$-feasible, proving the inductive step and concluding the proof.

\subsection{Conditions (\ref{eq:sufficient-cond-C-w}) are optimal when the weights are distinct}\label{subsec:main-thm-optimal}

To indicate the subtlety involved in verifying the optimality of (\ref{eq:sufficient-cond-C-w}), we start with an example showing that the construction of $\ul{C}^{\circ}$ in Section~\ref{sec:optimality} is not sufficient in genereal. Recall that our goal is to show that if we relax any of the conditions (\ref{eq:sufficient-cond-C-w}) then there exists a sequence $\ul{C}$ which satisfies the relaxed conditions and defines an infeasible problem.

\begin{example}\label{ex:C0-not-good}
 If $\ll = \ul{w}=(6,6,4,4,4,0)$ then $\ll' = (5^4,2^2)$, so $k=2$, $h_1=5$, $h_2=2$, $a_1=4$, $a_2=2$. We assume that $d\gg 0$ and consider the following table recording $\ul{C}^{\circ}-d\cdot\ul{w}$ and the numbers $b(\ul{w}^{\geq j})$.
\[
\setlength{\extrarowheight}{2pt}
\begin{array}{c|c|c|c|c|c|c}
j & 1 & 2  & 3 & 4 & 5 & 6 \\
 \hline 
C^{\circ}_j - d\cdot w_j & -1 & -1 & 3 & 1 & 1 & 3 \\ 
 \hline 
b(\ul{w}^{\geq j}) & 7 & 6 & 9 & 6 & 3 & 0 \\ 
\end{array}
\]
Recall that $\ul{C}^{\circ}$ was constructed to fail condition (\ref{eq:sufficient-cond-C-w}) for $i=1$, but in fact it also fails it for $i=2,3,4$. A~better choice of a capacity sequence is in this case to take
\begin{equation}\label{eq:modified-C0}
\setlength{\extrarowheight}{2pt}
\begin{array}{c|c|c|c|c|c|c}
j & 1 & 2  & 3 & 4 & 5 & 6 \\
 \hline 
C_j - d\cdot w_j & -1 & -5 & 5 & 3 & 1 & 3 \\ 
\end{array}
\end{equation}
which only fails condition (\ref{eq:sufficient-cond-C-w}) when $i=1$. To see that $\mf{BP}(d,\ul{C};\ul{w})$ is infeasible, suppose by contradiction that $\mc{B}_{\bullet}$ is a solution. Since $C_6=3<w_5$ it follows that $\mc{B}_6=\{0^d\}$. Since $C_5=4d+1$ and $6-4>1$, it follows that $\mc{B}_5=\{4^d\}$. Similarly,
\begin{itemize}
\item since $C_4=4\cdot d + 3$, it follows that $\mc{B}_4=\{6^u 4^{d-u}\}$ with $u\leq 1$;
\item since $C_3=4\cdot d + 5$, it follows that $\mc{B}_3=\{6^v 4^{d-v}\}$ with $v\leq 2$;
\item since $C_2=6\cdot d-5$ it follows that $\mc{B}_2=\{6^t 4^{d-t}\}$ with $t\leq d-3$;
\item since $C_1=6\cdot d-1$ it follows that $\mc{B}_1=\{6^x 4^{d-x}\}$ with $x\leq d-1$.
\end{itemize} 
The total number of balls of weight $6$ is then $u+v+t+x\leq 2d-1$, a contradiction with the fact that $w_1=w_2=6$ each have to appear $d$ times.
\end{example}

Suppose that the weights are distinct, consider any index $1\leq i_0\leq n$, and replace condition (\ref{eq:sufficient-cond-C-w}) for $i=i_0$ with
\[C_{i_0}+\cdots + C_n \geq d\cdot(w_{i_0}+\cdots+w_n) + b(\ul{w}^{\geq i_0})-1.\]
We claim that there exists a sequence $\ul{C}$ satisfying the new set of relaxed conditions for which $\mf{BP}(d,\ul{C};\ul{w})$ is not $(i_0-1)$-feasible, and in particular it is not feasible. By Lemma~\ref{lem:feasibility-truncations}, this is equivalent to finding $(C_{i_0},\cdots,C_n)$ so that $\mf{BP}(d,\ul{C}^{\geq i_0};\ul{w}^{\geq i_0})$ is not feasible and (\ref{eq:sufficient-cond-C-w}) holds for $i\geq i_0$, since we can then choose $C_1,\cdots,C_{i_0-1}$ sufficiently large so that conditions (\ref{eq:sufficient-cond-C-w}) are satisfied when $i<i_0$. This reduces the problem to the case $i_0=1$, in which case the tuple $\ul{C}^{\circ}$ constructed in Section~\ref{sec:optimality} can be used: we know that $\mf{BP}(d,\ul{C};\ul{w})$ is infeasible, so we only need to check that conditions (\ref{eq:sufficient-cond-C-w}) are satisfied for $i>1$, or equivalently, that
\begin{equation}\label{eq:ineq-delta-circ}
 \Delta^{\circ}_i + \cdots + \Delta^{\circ}_n \geq b(\ul{w}^{\geq i})\mbox{ for }i>1.
\end{equation}
Since $h_k=1$ (and $k=n-1$), we have by (\ref{eq:formulas-Ccirc}) that $\Delta^{\circ}_j = g_j(\ul{w})-1$ for $j\geq 2$. Moreover, since $\ul{w}$ is strictly decreasing, we have that
\[ b(\ul{w}^{\geq i}) = \sum_{j=i+1}^n (g_j(\ul{w})-1) = \sum_{j=i+1}^n \Delta^{\circ}_j,\]
from which (\ref{eq:ineq-delta-circ}) follows.

\begin{remark}\label{rem:when-C-is-C0}
 Another important case when the construction $\ul{C}^{\circ}$ proves the optimality of the conditions (\ref{eq:sufficient-cond-C-w}) is when (using the notation in Section~\ref{sec:optimality})
 \[a_1\leq a_2\leq\cdots\leq a_k.\]
 The interested reader can check that (\ref{eq:ineq-delta-circ}) holds in this case, so the proof of the optimality of (\ref{eq:sufficient-cond-C-w}) follows as in the case of distinct weights.
\end{remark}

\section{Open questions}\label{sec:open}

The main question left open by this work is that of the optimality of the conditions (\ref{eq:sufficient-cond-C-w}) in Theorem~\ref{thm:sufficient-condition}. To formulate it precisely, we need some care in avoiding trivial counterexamples. For instance, when $n=2$ and $w_1=w_2=w$ the conditions (\ref{eq:sufficient-cond-C-w}) become
\[ C_2\geq d\cdot w,\quad C_1+C_2 \geq 2\cdot d\cdot w.\]
Since $C_1\geq C_2$, the first condition implies the second, so relaxing the second condition to $C_1+C_2 \geq 2\cdot d\cdot w-1$ leads to an equivalent set of conditions. Therefore, we call a \defi{strict relaxation} of (\ref{eq:sufficient-cond-C-w}) one for which the resulting set of solutions is strictly larger. With this convention, the optimality question becomes.

\begin{question}\label{que:optimality}
 Is it true that for an arbitrary $\ul{w}$, if we strictly relax the inequalities (\ref{eq:sufficient-cond-C-w}) then for $d\gg 0$ there exists a tuple of capacities $\ul{C}$ for which $\mf{BP}(d,\ul{C};\ul{w})$ is infeasible?
\end{question}

By generalizing the construction in Example~\ref{ex:C0-not-good} it can be shown that Question~\ref{que:optimality} has a positive answer when there are at most three distinct weights $w_i$. In view of Remark~\ref{rem:when-C-is-C0}, a first interesting case to consider is when $k=3$ and $a_1>a_2>a_3$. For instance, one could start by analyzing the following.

\begin{question}
 Are the inequalities (\ref{eq:sufficient-cond-C-w}) optimal for $\ul{w}=(10,9,9,6,6,0)$?
\end{question}

One can consider more generally the collection of all tuples $(b_1,\cdots,b_n)$ for which the conditions
\begin{equation}\label{eq:general-feasibility}
 C_i+\cdots+C_n \geq d\cdot(w_i+\cdots+w_n) + b_i
\end{equation}
guarantee that $\mf{BP}(d,\ul{C};\ul{w})$ is feasible when $d\gg 0$. These tuples form a poset ideal with respect to the natural partial order where $(b_1,\cdots,b_n)\leq(c_1,\cdots,c_n)$ if $b_i\leq c_i$ for all $i$. 

\begin{question}\label{que:optimal-tuples}
 What is the structure of the minimal elements in this poset ideal?
\end{question}

A different optimality question is concerned with finding the best bounds for $d$ with respect to the input data so that our conditions guarantee feasibility. 

\begin{problem}
 Determine the order of magnitude of optimal bounds $d_{\ul{w}}$ and $d_{\ul{w}}^1$ for Theorems~\ref{thm:sufficient-condition} and~\ref{thm:inductive-feasibility}.
\end{problem}

For the classical partitioning problem it is not usual to ask that each of the bins $\mc{B}_i$ contains the same number of balls, or that the number of balls coincides with the number of bins. A modification of $\mf{BP}(d,\ul{C};\ul{w})$ arises then by decoupling the number of weights from the number of bins, and considering instead the problem of assigning a collection containing $d_i$ balls of weight $w_i$ for $i=1,\cdots,m$, to $n$ bins $\mc{B}_1,\cdots,\mc{B}_n$, where $m,n$ and $w_i$ are fixed, and $d_1,\cdots,d_m\to\infty$, without restricting the number of balls that go into each bin.

\begin{problem}
 Find asymptotically optimal feasibility conditions for the partitioning problem with $m$ weights, $n$ bins, when $d_1,\cdots,d_m\gg 0$.
\end{problem}

\section*{Acknowledgments} 
The author would like to thank Bernd Sturmfels and Jens Vygen for helpful suggestions regarding the literature on partitioning and bin packing problems. Experiments with the computer algebra software Macaulay2 \cite{M2} have provided numerous valuable insights. The author acknowledges the support of the Alfred P. Sloan Foundation, and of the National Science Foundation Grant No.~1901886.


\end{document}